\documentclass[reqno]{amsart}
\usepackage{graphicx,amsmath,amssymb,amsthm}

\newtheorem{thm}{Theorem}
\newtheorem{lem}[thm]{Lemma}
\theoremstyle{definition}
\newtheorem{defn}[thm]{Definition}
\theoremstyle{remark}
\newtheorem{rem}[thm]{Remark}
\newtheorem{prop}[thm]{Proposition}
\newtheorem{cor}[thm]{Corollary}
\newtheorem{remark}[thm]{Remark}
\newtheorem{example}[thm]{Example}
\newcommand{\RR}{\mathbb{R}}
\newcommand{\CC}{\mathbb{C}}
\newcommand{\pts}{\mathcal P}
\newcommand{\incidences}{\mathcal I}
\newcommand{\sphrs}{\mathcal S}
\newcommand{\BZ}{\mathbf{Z}}
\newcommand{\reali}{\operatorname{Reali}}
\begin{document}
\title[An improved bound on point-surface incidences in $\RR^3$]{An improved bound on the number of point-surface incidences in three dimensions}
\author[J.~Zahl]{Joshua Zahl}%
\address{Department of Mathematics, UCLA, Los Angeles CA 90095-1555, USA}
\email{jzahl@math.ucla.edu}

\date{\today}
\begin{abstract} We show that $m$ points and $n$ smooth algebraic surfaces of bounded degree in $\mathbb{R}^3$ satisfying suitable nondegeneracy conditions can have at most $O(m^{\frac{2k}{3k-1}}n^{\frac{3k-3}{3k-1}}+m+n)$ incidences, provided that any collection of $k$ points have at most $O(1)$ surfaces passing through all of them, for some $k\geq 3$. In the case where the surfaces are spheres and no three spheres meet in a common circle, this implies there are $O((mn)^{3/4} + m +n)$ point-sphere incidences. This is a slight improvement over the previous bound of $O((mn)^{3/4} \beta(m,n)+ m +n)$ for $\beta(m,n)$ an (explicit) very slowly growing function. We obtain this bound by using the discrete polynomial ham sandwich theorem to cut $\mathbb{R}^3$ into open cells adapted to the set of points, and within each cell of the decomposition we apply a Turan-type theorem to obtain crude control on the number of point-surface incidences. We then perform a second polynomial ham sandwich decomposition on the irreducible components of the variety defined by the first decomposition. As an application, we obtain a new bound on the maximum number of unit distances amongst $m$ points in $\mathbb{R}^3$.
\end{abstract}
\maketitle
\section{Introduction}
In \cite{Clarkson}, Clarkson, Edelsbrunner, Guibas, Sharir, and Welzl obtained the following bound on the number of incidences between points and spheres in $\RR^3$:
\begin{thm}[Clarkson et al.]\label{clarksonThm}
The number of incidences between $m$ points and $n$ spheres in $\RR^3$ with no three spheres meeting at a common circle is
\begin{equation}\label{clarksonBound}
O( (mn)^{3/4}\beta(m,n) + m + n),
\end{equation}
where $\beta(m,n)$ is a very slowly growing function of $m$ and $n$. In particular, $\beta(m,n)\leq 2^{C\alpha(m^3/n)^2},$ where $\alpha(s)$ is the inverse Ackerman function and $C$ is a large constant.
\end{thm}
We obtain the following slight sharpening:
\begin{thm}\label{mainThm}
Let $k\geq 3,$ and let $\pts\subset\RR^3$ be a collection of $m$ points and $\sphrs$ a collection of $n$ smooth algebraic surfaces of bounded degree (the degree is allowed to depend on $k$) such that for some constant $C$ we have $|S\cap S^\prime\cap S^{\prime\prime}|\leq C$ for all $S,S^\prime,S^{\prime\prime}\in\sphrs$, and for any collection of $k$ points in $\RR^3$, there are at most $C$ surfaces that contain all $k$ points. Then the number of incidences between points in $\pts$ and surfaces in $\sphrs$ is
\begin{equation}\label{mySurfaceBound}
O(m^{\frac{2k}{3k-1}}n^{\frac{3k-3}{3k-1}}+m+n),
\end{equation}
where the implicit constant depends only on $k$, $C$, and the degree of the algebraic surfaces.

In particular, the number of incidences between $m$ points and $n$ spheres in $\RR^3$ with no three spheres meeting at a common circle is
\begin{equation}\label{mySphereBound}
O( (mn)^{3/4} + m + n).
\end{equation}
\end{thm}
\begin{rem}
 The requirements that every three surfaces meet in $C$ points and that every $k$ points have at most $C$ surfaces passing through them are analogous to the definition of ``curves with $k$ degrees of freedom'' from \cite{Pach}, though in \cite{Pach} the curves do not need to be algebraic.
\end{rem}

\begin{rem}
The requirement that every three surfaces meet in a complete intersection, or some variant thereof, is necessary to prevent the situation in which all of the surfaces meet in a common curve and all of the points lie on that curve, yielding $mn$ incidences (i.e.~if we don't place any restrictions on how the surfaces can intersect, then the trivial bound of $mn$ incidences is sharp).
\end{rem}
\begin{rem}
Theorem \ref{clarksonThm} can be extended to the more general case of bounded degree algebraic surfaces using the decomposition techniques described in \cite[\S 8.3]{Sharir} to obtain an
analogue of \eqref{mySurfaceBound}. Doing so yields a bound of $O(m^{\frac{2k}{3k-1}}n^{\frac{3k-3}{3k-1}}\beta(m,n)+m+n)$ for $\beta$ a slowly growing function.
\end{rem}

\subsection{Previous results}
Similar results to Theorem \ref{clarksonBound} and \ref{mySurfaceBound} have been obtained by \L{}aba and Solymosi in \cite{Laba} and by Iosevich, Jorati, and \L{}aba in \cite{Iosevich}. In \cite{Laba} and \cite{Iosevich}, however, the authors consider a more general class of surfaces (they need not be algebraic), but they require that the point set be ``homogeneous'' in a suitable sense.

Our techniques do not work well when $k=2$, i.e.~for obtaining bounds on point-hyperplane incidences, but this case has been studied by other authors (see e.g.~\cite{Elekes}, where the authors obtain sharp bounds on point-hyperplane incidences under a slightly different set of non-degeneracy conditions).

\subsection{Update 7/4/2011} The author has recently become aware that concurrently with this paper, Kaplan, Matou\v{s}ek, Safernov\'a, and Sharir in \cite{Kaplan2} obtained results similar to the bound \eqref{mySphereBound} using similar methods. Kaplan et.~al.~are able to avoid some of the technical difficulties present in this paper by using an explicit paramaterization of the sphere by rational functions.

\subsection{Proof sketch}
Clarkson et al.~obtain Theorem \ref{clarksonThm} through their ``Canham threshold plus divide and conquer'' technique: the arrangement of spheres in $\RR^3$ is subdivided into smaller collections through a careful partitioning of $\RR^3$, and the number of incidences between these smaller collections of spheres and points is controlled by a Turan-type bound on the number of edges in a bipartite graph with certain forbidden subgraphs.

In this paper, we employ similar ideas, except instead of dividing the problem into smaller subproblems by partitioning $\RR^3$ into cells using a decomposition adapted to the collection of spheres (or more general nonsingular algebraic surfaces), we employ a partition adapted to the collection of points. This partition is obtained from the discrete polynomial ham sandwich theorem recently used to great effect by Guth and Katz in \cite{Guth} and more recently by Solymosi and Tao in \cite{Solymosi} and by Kaplan, Matou\v{s}ek, and Sharir in \cite{Kaplan}. Specifically, we find a polynomial $P$ such that the complement of the zero set of $P$ consists of open ``cells,'' none of which contain too many points. We can then apply a Turan-type bound to the points and surfaces inside each cell. However, some points may lie on the zero set of $P$, and thus do not lie in any of the cells. To deal with these points, we perform a second polynomial ham sandwich decomposition to find a polynomial $Q$ whose zero set partitions the zero set of $P$ into cell-like objects, and we apply the Turan-type bound to each of these ``cells.'' While it is possible that a point could lie in the zero set of both $P$ and $Q$, we can use B\'ezout-type theorems to control how often this can occur.
\subsection{Some difficulties with real algebraic sets}\label{obstructions}
There are several technical difficulties that have to be dealt with while executing the above strategy. In contrast to the situation over $\CC$, there exist polynomials $P_1,\ldots,P_d\in\RR[x_1,\ldots,x_d]$ of degrees $D_1,\ldots,D_d$ such that $\{P_1=0\}\cap\ldots\cap\{P_d=0\}$ contains more than $D_1\ldots D_d$ isolated points, i.e.~the naive analogue of B\'ezout's theorem fails over $\RR$. To deal with this problem, we will sometimes be forced to embed our varieties into $\CC$ and use the (usual) B\'ezout's theorem (though we have to be careful that the intersection of the embedded varieties does not contain new, unexpected components of positive dimension).

A second difficulty concerns the failure of the Nullstellensatz for varieties defined over $\RR$. In contrast to the complex case, If $(P)$ is a principal prime ideal and $Q$ is a real polynomial, it need not be the case that if $Q$ vanishes identically on $\{x\in\RR^d\colon P=0\}$ then $Q\in(P)$. Luckily, there is a special type of ideal known as a ``real ideal'' for which an analogue of the Nullstellensatz does hold. Frequently we will be required to replace our polynomials with new polynomials that generate real ideals.

Finally, if $P\in\RR[x_1,\ldots,x_d]$ then the dimension of $\{x\in\RR^d\colon P=0\}$ may be less than $d-1,$ and even if $P$ is squarefree, $\nabla P$ may vanish on $\{P=0\}$. Again, we can remedy this problem by working with (irreducible) polynomials that generate real ideals.
\subsection{Thanks}
The author is very grateful to Jordan Ellenberg, Larry Guth, Netz Katz, Jozsef Solymosi, and Terence Tao for helpful discussions, and to Haim Kaplan, Ji\v{r}\'i Matou\v{s}ek, Zuzana Safernov\'a, and Micha Sharir for pointing out errors in an earlier version of this paper. Finally, the author would like to thank the anonymous referee for his or her careful proofreading and corrections. The author was supported in part by the Department of Defense through the National Defense Science \& Engineering Graduate Fellowship (NDSEG) Program.
\section{Main Result}
\subsection{Notation}
Throughout the paper, $c$ and $C$ will denote sufficiently small and large constants, respectively, which are allowed to vary from line to line. We will write $A\lesssim B$ to mean $A<CB,$ we will write $A\sim B$ to mean $cB<A<CB,$ and we say that a quantity is $O(A)$ if it is $\lesssim A$.

Let $\sphrs$ be a collection of smooth (real) surfaces and $\pts$ a collection of points. Then $\incidences(\pts,\sphrs)$ is the number of incidences between the surfaces in $\sphrs$ and the points in $\pts$. If $S\in\sphrs$ is a surface, then $f_S$ is the polynomial whose zero set is $S$.

All ideals and varieties will be assumed to be affine. Unless otherwise specified, all ideals are subsets of $\RR[x_1,\ldots,x_d],$ and all varieties are defined over $\RR$ and thus are subsets of $\RR^d$, though sometimes we will specialize to the case $d=3$. If $P$ is a polynomial, $(P)\subset\RR[x_1,\ldots,x_d]$ is the ideal generated by $P$.

Special emphasis will be placed on ``real ideals.'' These are described in Definition \ref{defnOfRealIdeal} of Appendix \ref{algebraicGeometryDefinitions}, and they should not be confused with ideals that are merely subsets of $\RR[x_1,\ldots,x_d]$. On the other hand, a ``real variety'' is merely a variety defined over $\RR$ (as opposed to $\CC$).

If $I$ is an ideal, we use
\begin{equation*}
\BZ(I)=\{x\in\RR^d\colon P(x)=0\ \textrm{for all}\ P\in I\}
\end{equation*}
to denote the zero set of $I$. If $P$ is a polynomial we shall abuse notation and use $\BZ(P)$ to denote $\BZ((P))=\{x\in\RR^d\colon P(x)=0\}$. If $Z\subset\RR^d$ is a real variety, then we define
\begin{equation*}
\mathbf I(Z) = \{P\in\RR[x_1,\ldots,x_d]\colon P(x) = 0\ \textrm{for all}\ x\in Z\}
\end{equation*}
to be the ideal of polynomials that vanish on $Z$.

If $Z\subset\RR^d$ is a real variety, then $Z^*\subset\CC^d$ denotes the smallest complex variety containing $Z$. Conversely, if $Z\subset\CC^d$ is a complex variety, then $\mathfrak{R}(Z)\subset\RR^d$ is its set of real points. 

If $\mathcal Q\subset\RR[x_1,\ldots,x_d]$ is a collection of real polynomials, then we can partition $\RR^d\backslash\bigcup_{Q\in\mathcal Q}\BZ(Q)$ into a collection of open sets such that on each open set, the polynomials from $\mathcal Q$ do not change sign. These sets will be called the \emph{realizations of realizable strict sign conditions of $\mathcal Q$}. Similarly, if $Z\subset\RR^d$ is a variety, then we can consider the restriction of the above open sets to $Z$, and these are called the \emph{realizations of realizable strict sign conditions of $\mathcal Q$ on $Z$}. These notions are defined more precisely in Appendix \ref{algebraicGeometryDefinitions}.

\subsection{Preliminaries}
Following \cite{Clarkson}, we shall need the following Turan-type bound:
\begin{thm}[K\H{o}vari, S\'os, Turan \cite{Turan}] Let $s,t$ be fixed, and let $G=G_1\sqcup G_2$ be a bipartite graph with $|G_1|=m,\ |G_2|=n$ that contains no copy of $K_{s,t}$. Then $G$ has at most $O(nm^{1-1/s} + m)$ edges. Symmetrically, $G$ has at most $O(mn^{1-1/t} + n)$ edges. All implicit constants depend only on $s$ and $t$.
\end{thm}

In our case, we have that $|S\cap S^\prime\cap S^{\prime\prime}|\leq C$ for every three surfaces $S,S^\prime,S^{\prime\prime}$, and any $k$ points have at most $C$ surfaces passing through all of them. Thus we have the bounds

\begin{align}
\incidences(\pts,\sphrs)&\lesssim |\pts||\sphrs|^{1-1/k}+|\sphrs|,
\label{canhamThreshholdForSpheres1}\\
\incidences(\pts,\sphrs)&\lesssim |\pts|^{2/3}|\sphrs|+|\pts|.
\label{canhamThreshholdForSpheres2}
\end{align}

Recall the discrete polynomial partitioning theorem from \cite{Guth}:
\begin{thm}\label{hamSandwichThm}
Let $\pts$ be a collection of points in $\RR^d$, and let $D>0$. Then there exists a non-zero polynomial $P$ of degree at most $D$ such that each connected component of $\RR^d\backslash\BZ(P)$ contains at most $O(|\pts|/D^d)$ points of $\pts$.
\end{thm}
\begin{rem}
Without loss of generality, we can assume that $P$ is squarefree. Indeed if $P$ is not squarefree then we can replace $P$ by its squarefree part, and the new polynomial still has all of the desired properties.
\end{rem}

\begin{example}\label{polyDecompEx}
Consider the set of 24 points
\begin{equation*}
\pts_1=\{ (0,\pm 1,\pm 1),(0,\pm2,\pm2), (\pm 1,\pm1,\pm1), (\pm2,\pm2,\pm2)\}\subset\RR^3,
\end{equation*}
and let $D=3$. Then the polynomial $P_1(x_1,x_2,x_3)=x_1x_2x_3$ partitions $\RR^3$ into 8 octants, each of which contains 2 points from $\pts_1$.
\end{example}
\begin{rem}
Note as well that in the above example, the 8 points $\{0,\pm1,\pm1\},\ \{0,\pm2,$ $ \pm2\}$ lie on the set $\BZ(P_1)$ and thus they do not lie inside any of the open components of $\RR^3\backslash\BZ(P_1)$. This is not merely a consequence of us choosing $P_1$ poorly; it is an unavoidable phenomena that occurs when performing the discrete polynomial partitioning decomposition. In order to control the number of incidences between points lying on $\BZ(P_1)$ and surfaces in $\sphrs$, we shall have to perform a second polynomial partitioning decomposition ``on'' the surface $\BZ(P_1)$. For technical reasons, we cannot simply consider the complement of our ham sandwich ``cut'' as a union of relatively open subsets of $\BZ(P_1)$. Instead, we need to perform a somewhat more detailed decomposition that partitions $\BZ(P_1)$ into sets that are realizations of realizable strict sign conditions of a certain family of polynomials. This is made precise in the theorem below. See Appendix \ref{algebraicGeometryDefinitions} for the definition of a real ideal, a strict sign condition, and the realization of a strict sign condition.
\end{rem}

\begin{thm}[Discrete polynomial partitioning theorem on a hypersurface]\label{variantHamSandwichThm}
Let $\pts$ be a collection of points in $\RR^d$ lying on the set $Z=\BZ(P)$ for $P$ an irreducible polynomial of degree $D$ such that $(P)$ is a real ideal. Let $\rho>0$ be a small constant, and let $E\geq \rho D$. Then there exists a collection of polynomials $\mathcal Q\subset \RR[x_1,\ldots,x_d]$ with the following properties:
\begin{enumerate}
 \item\label{variantHamSandwichThmItm1} $|\mathcal Q|\leq \log_2 (DE^{d-1})+O(1)$.
 \item\label{variantHamSandwichThmItm2} $\sum_{\mathcal Q}\deg Q \lesssim E$.
 \item\label{variantHamSandwichThmItm3} None of the polynomials in $\mathcal Q$ vanish identically on $Z$.
 \item\label{variantHamSandwichThmItm4} The realization of each of the $O(DE^{d-1})$ strict sign conditions of $\mathcal Q$ on $Z$ contains at most $O(\frac{|\pts|}{DE^{d-1}})$ points of $\pts$.
\end{enumerate}
All implicit constants depend only on $\rho$ and the dimension $d$.
\end{thm}
We shall defer the proof of Theorem \ref{variantHamSandwichThm} to Appendix \ref{variantHamSandwichThmAppendix}. In our applications, we will always have $d=3$.
\begin{example}\label{polyDecompEx2}
Let us continue Example \ref{polyDecompEx}. The polynomial $P_1$ from Example \ref{polyDecompEx} was not irreducible, but we can factor it into the three irreducible factors $x_1,x_2,x_3$. All of the points lying on $\BZ(P_1)$ actually lie on the irreducible component $\BZ(x_1)$, so we let $P_2(x_1,x_2,x_3) = x_1$. Note that $(P_2)=(x_1)$ is a real ideal and let $D=\deg(P_2)=1$. Select $E=2$ (which is larger than $D$). Then the collection of polynomials $\mathcal Q=\{x_2,x_3\}$ satisfies the requirements of Theorem \ref{variantHamSandwichThm}. The realizations of realizable strict sign conditions of $\mathcal Q$ on $Z$ are the 4 sets of the form
\begin{equation}
\{(x_1,x_2,x_3)\colon x_1=0,\ \pm x_2>0,\ \pm x_3>0\}.
\end{equation}
Note that each of these sets contains $2$ points of $\pts_1\cap\BZ(P_2)$. Two coincidences occur in this example that are not present in general. First, in this example the realizations of the four strict sign conditions of $\mathcal Q$ on $Z$ correspond to the four connected components of $Z\backslash\bigcup_{\mathcal Q}\BZ(Q)$. In general, each realization of a strict sign condition may be a union of multiple connected components of $Z\backslash\bigcup_{\mathcal Q}\BZ(Q)$. Second, each of the polynomials in $\mathcal Q$ were irreducible factors of $P_1$. In general this does not occur.
\end{example}

We are now ready to prove Theorem \ref{mainThm}.

\begin{center}
\end{center}
\subsection{Proof of Theorem \ref{mainThm}}
\begin{proof} Let $\sphrs$ and $\pts$ be as in the statement of Theorem \ref{mainThm}. From \eqref{canhamThreshholdForSpheres1} and \eqref{canhamThreshholdForSpheres2}, we have that if $n>cm^k$ or $m> cn^3$ for some fixed small constant $c>0$ to be specified later, then Theorem \ref{mainThm} immediately holds. Thus we may assume
\begin{equation}\label{mAndMNotTooDifferent}
\begin{split}
n&< cm^k,\\
m&<cn^3.
\end{split}
\end{equation}

Let $P$ be a squarefree polynomial of degree at most $D$ ($D$ will be determined later, but the impatient reader can jump to \eqref{valueOfD}) that cuts $\RR^3$ into $\sim D^3$ cells with $O(m/D^3)$ points in each cell, and let $Z=\BZ(P)$. Let $m_i$ be the number of points lying in the $i$--th cell of the above decomposition, and let $n_i$ be the number of surfaces that meet the interior of the $i$--th cell.
\begin{lem}\label{sumOfNiLemma}
\begin{equation}
\sum n_i\lesssim D^2 n.
\end{equation}
\end{lem}
\begin{proof}
Let $S$ be a surface that is not contained in $Z$ and is not entirely contained in the closure of one cell. Since there are finitely many cells, we can select a large closed ball $B\subset\RR^3$ so that the number of cells that meet $S$ is equal to the number of cells that meet $S\cap B.$ We can apply a small generic translation to $S$, and doing so can only increase the number of cells that meet $S\cap B$ (and thus can only increase the number of cells that meet $S$). Select a generic vector $v\in R^3$ and let $T(x) = v \wedge \nabla f_S(x) \wedge \nabla P(x),$ so if $x\in S\cap Z$ and $\nabla f_S(x)$ and $\nabla P(x)$ are non-zero and non-collinear, then $T(x)=0$ if the curve $S\cap Z$ is tangent at $x$ to a plane with normal vector $v$.

For every cell $\Omega$ that meets $S$, there is a point $x \in \partial \Omega\cap S$ satisfying the following properties.
\begin{enumerate}
\item $x$ is a smooth point of the space curve $Z\cap S$.
\item\label{sumOfNiLemmaItm2} $x$ is a non-singular intersection point of $\BZ(T)\cap Z\cap S$.
\item\label{sumOfNiLemmaItm3} $x$ is a smooth point of $\partial\Omega$.
\end{enumerate}
These three properties follow from the fact that $v$ is generic and we picked a generic translation of $S$. From Item \ref{sumOfNiLemmaItm3}, each point $x$ satisfying the above properties can be associated to at most 2 distinct cells $\Omega,\Omega^\prime$.
By Item \ref{sumOfNiLemmaItm2} and the real B\'ezout inequality (see e.g.~\cite[\S 4.7]{Basu}), there can be at most $\deg(P)\deg(T)\deg(f_s) = O(D^2)$ such points, and thus $S$ can enter at most $O(D^2)$ such cells. Since there are $n$ surfaces $S\in\sphrs,$ the result follows.
\end{proof}

Using Lemma \ref{sumOfNiLemma} and the bound from \eqref{canhamThreshholdForSpheres1} we can control the number of incidences between points not lying in $Z$ and surfaces in $\sphrs$:
\begin{equation}\label{controlOfLevelOneCells}
\begin{split}
\incidences(\pts \backslash Z, \sphrs)&= \sum_i \incidences(\pts \cap \Omega_i,\mathcal S)\\
&\lesssim\sum_i m_i n_i^{1-1/k} + n_i\\
&\lesssim \Big(\sum_i m_i^k\Big)^{1/k}\Big(\sum_i n_i\Big)^{1-1/k}+ D^2n\\
&\lesssim \Big(D^3 \frac{m^k}{D^{3k}}\Big)^{1/k}(D^2n)^{1-1/k}+D^2 n\\
&\lesssim \frac{mn^{1-1/k}}{D^{1-1/k}}+D^2 n.
\end{split}
\end{equation}

We must now control $\incidences (\pts \cap Z, \mathcal S)$. We have
\begin{equation}\label{decompositionOfSpheres}
\incidences (\pts \cap Z, \sphrs) = \incidences(\pts \cap Z, \sphrs_1) + \incidences(\pts \cap Z,\sphrs_2),
\end{equation}
where $\sphrs_1$ is the set of surfaces contained in $Z$, and $\sphrs_2$ are the remaining surfaces. Since $Z$ has degree $D$, $Z$ can contain at most $D$ surfaces from $\sphrs$, i.e.~$|\sphrs_1|\leq D$. By \eqref{canhamThreshholdForSpheres2},
\begin{equation}\label{controlOfLevelOneEmbeddedSpheres}
\begin{split}
\incidences(\pts \cap Z,\sphrs_1) &\lesssim  |\sphrs_1|\ |\pts|^{2/3}+|\pts|\\
&\lesssim Dm^{2/3}+m.
\end{split}
\end{equation}

Thus it remains to control $\incidences(\pts \cap Z, \mathcal S_2)$. Write $P=P_1 \ldots P_\ell$, where each $P_j$ is irreducible of degree $D_j$, and let $Z_j=\BZ(P_j)$. Thus we have $D_1+\ldots+D_\ell \leq D$, and $Z=\bigcup Z_j.$ We would like to use Lemma \ref{variantHamSandwichThm} perform a second discrete polynomial ham sandwich decomposition on each variety $Z_j$, but if $(P_j)$ is not a real ideal then we cannot apply the lemma. Luckily, the following lemma lets us remedy this situation.
\begin{lem}\label{allPolysGenerateRealIdeals}
 Let $\mathcal A \subset\RR[x_1,\ldots,x_d]$ be a collection of irreducible polynomials. Then we can find a new collection $\mathcal A^\prime$ of irreducible polynomials such that:
\begin{enumerate}
 \item\label{allPolysGenerateRealIdealsItm1} $\bigcup_{P\in\mathcal A}\BZ(P)\subset\bigcup_{P\in\mathcal A^\prime}\BZ(P)$.
 \item $\sum_{P\in\mathcal A}\deg P \leq \sum_{P\in \mathcal A^\prime}\deg P$.
 \item\label{allPolysGenerateRealIdealsItm3}  $(P)$ is a real ideal for each $P\in\mathcal A^\prime.$
\end{enumerate}
\end{lem}
\begin{proof}
We shall proceed by induction on $\sum_{P\in\mathcal A}\deg P$. If the sum is 1 then the result is trivial since in that case $\mathcal A$ consists of a single linear polynomial, so we can let $\mathcal A^\prime =\mathcal A$. Suppose the lemma has been established for all families $\tilde{\mathcal A}$ with $\sum_{P\in\tilde{\mathcal A}}\deg P< w$, and let $\sum_{P\in\mathcal A}\deg P = w$. If $(P)$ is a real ideal for every $P\in\mathcal A$ then the result is immediate. If not, select $P\in\mathcal A$ such that $(P)$ is not a real ideal. By Proposition \ref{propertiesOfPrincipleRealIdealProp} in Appendix \ref{propOfRealVarietiesAppendix}, $\nabla P$ vanishes on $\BZ(P)$. Let $v\in\RR^d$ be a generic unit vector. Then $\BZ(P)\subset \BZ(\nabla_v P)$ and $\deg(\nabla_v P)<\deg P.$ Write $\nabla_v P=Q_1\ldots Q_a$ as a product of irreducible components, and let $\tilde{\mathcal A} = \mathcal A \cup\{Q_1,\ldots,Q_a\}\backslash \{P\}.$ We have $\sum_{P\in\tilde{\mathcal A}}\deg P < \sum_{P\in\mathcal A}\deg P=w$, and $\bigcup_{P\in\mathcal A}\BZ(P)\subset \bigcup_{P\in\tilde{\mathcal A}}\BZ(P)$. Apply the induction hypothesis to $\tilde{\mathcal A}$ to obtain a family $\tilde{\mathcal A}^\prime$ satisfying Properties \ref{allPolysGenerateRealIdealsItm1}--\ref{allPolysGenerateRealIdealsItm3} with $\tilde{\mathcal A}$ in place of $\mathcal A$. We can verify that $\tilde{\mathcal A}^\prime$ has the desired properties.
\end{proof}
After applying Lemma \ref{allPolysGenerateRealIdeals}, we can assume that each irreducible polynomial $P_j$ in the decomposition of $P$ generates a real ideal. Write $\pts \cap Z = \bigsqcup \pts_j$, where $\pts_j$ consists of those points lying in $Z_j$. If a point lies on two or more such varieties, place it into only one of the sets. We need to distinguish between several cases. Let
\begin{align*}
\mathcal A_1 &= \{j\colon |\pts_j|^k < D_j^kn\},\\
\mathcal A_2 &= \{j\colon D_j^kn \leq |\pts_j|^k < c D_j^{3k-1}n\},\\
\mathcal A_3 &= \{j\colon  |\pts_j|^k \geq cD_j^{3k-1}n\},\\
\end{align*}
where $c$ is a small constant to be determined later. For each $j\in\mathcal A_1$ we have
\begin{equation}\label{controlOfIncidencesFromA1}
\begin{split}
\incidences(\pts \cap Z_j,\sphrs_2)&\lesssim |\pts_j|n^{1-1/k}+n\\
&\lesssim D_jn,
\end{split}
\end{equation}
where the second inequality uses the assumption $|\pts_j|< D_jn^{1/k}$. Summing \eqref{controlOfIncidencesFromA1} over all
$j\in\mathcal A_1$, we obtain
\begin{equation}\label{controlOfIncidencesFromA1Summed}
\begin{split}
\incidences(\pts \cap \bigcup_{A_1}Z_j,\sphrs_2)&\lesssim
\sum_{\mathcal A_1} D_jn \\
&\leq Dn.
\end{split}
\end{equation}

Now we must control the incidences between surfaces and points lying on varieties $Z_j,\ j\in\mathcal A_2$ or $\mathcal A_3$. If $j\in\mathcal A_2$, use Theorem \ref{hamSandwichThm} to select a squarefree polynomial $Q_j$ of degree at most $E_j$,
\begin{equation}\label{defnOfE}
E_j=\Big(\frac{|\pts_j|^k}{nD_j^k}\Big)^{1/(2k-1)},
\end{equation}
that cuts $\RR^3$ into $\sim E_j^3$ cells, each of which contains $\lesssim |\pts_j|/ E_j^3$ points of $\mathcal P_j$. Recall that $P_j$ is irreducible, $(P_j)$ is real, and $j\in\mathcal A_2$ implies $\deg(Q_j)\leq E_j <\deg(P_j)$. Thus $Q_j$ does not vanish identically on $Z_j$. Let $\mathcal Q_j = \{Q_j\}$ and let $W_j=\BZ(Q_j)$.

If $j\in\mathcal A_3$, let $E_j$ be as in \eqref{defnOfE} and use Theorem \ref{variantHamSandwichThm} (with $E=E_j$) to find a family $\mathcal Q_j$ of polynomials satisfying properties \ref{variantHamSandwichThmItm1}--\ref{variantHamSandwichThmItm4} of the theorem. In particular, the realizations of the realizable strict sign conditions of $\mathcal Q_j$ on $Z_j$ partition $Z_j$ into ${\sim D_j E_j^2}$ (not necessarily connected) sets, each of which contains $\lesssim |\pts_j|/D_j E_j^2$ points, plus the ``boundary'' $Z_j\cap \bigcup_{\mathcal Q_j}\BZ(Q).$ Define $W_j=\bigcup_{\mathcal Q_j}\BZ(Q)$ (thus the definition of $W_j$ depends on whether $j\in\mathcal A_2$ or $j\in\mathcal A_3$).

Regardless of whether $j\in \mathcal A_2$ or $\mathcal A_3$, have
\begin{equation}\label{secondLevelDecomposition}
 \incidences(\pts_j , \sphrs_2)=\ \incidences(\pts_j \backslash W_j, \sphrs_2) + \incidences(\pts_j\cap W_j,\sphrs_2).
\end{equation}

We shall begin by bounding the first term of \eqref{secondLevelDecomposition}. If $j\in \mathcal A_2$, then through the same computation performed in \eqref{controlOfLevelOneCells} we have
\begin{equation}\label{2ndLevelA2InsideCell}
\begin{split}
\incidences(\pts_j\backslash W_j, \sphrs_2) &\lesssim \frac{|\pts_j|n^{1-1/k}}{E_j^{1-1/k}} + nE_j^2\\
&\leq \frac{|\pts_j|n^{1-1/k}}{E_j^{1-1/k}} + nD_jE_j.
\end{split}
\end{equation}

If $j\in\mathcal A_3$, then let $\Omega_{ij}$ be the realization of the $i$--th realizable strict sign condition of $\mathcal Q_j$ on $Z_j$. Let $m_{ij}=|\pts_j\cap\Omega_{ij}|$, and let $n_{ij}$ be the number of surfaces in $\sphrs_2$ that intersect $\Omega_{ij}$.
\begin{lem}\label{controlOfSphereLevelTwoCellIntersectionsLemma}
\begin{equation} \label{controlOfSphereLevelTwoCellIntersections}
\sum_i n_{ij}\lesssim nD_jE_j.
\end{equation}
\end{lem}
\begin{proof}
If a surface $S\in\sphrs_2$ lies in $W_j$ then it does not contribute to the above sum, so we need only consider those surfaces $S$ that do not lie in $Z_j$ or $W_j$. First, we can replace each $Q\in\mathcal Q$ by the polynomial $Q+\epsilon$ for $\epsilon>0$ a sufficiently small constant. If $S\cap \{x\in\RR^3\colon Q(x)>0\}\cap Z_j\neq\emptyset$, then there must be a point on $S\cap Z_j$ where $Q$ is positive, so $S\cap \{x\in\RR^3\colon Q(x)+\epsilon>0\}\cap Z_j\neq\emptyset$ for $\epsilon$ sufficiently small, and similarly for $S\cap \{x\in\RR^3\colon Q(x)<0\}\cap Z_j$. Thus replacing each $Q\in\mathcal Q$ by $Q+\epsilon$ does not increase the number of realizations of realizable strict sign conditions that meet $S$. We shall select a small generic (with respect to $S$ and $Z_j$) choice of $\epsilon$.

By Corollary \ref{surfaceEnteringStrictSignConditions} in Appendix \ref{propOfRealVarietiesAppendix}, we can assume that each irreducible component of each polynomial in $\mathcal{Q}_j$ generates a real ideal.

Instead of counting $\sum_i n_{ij}$ directly, we shall bound the number of times a surface $S$ enters a connected component of $Z_j\backslash W_j$, as this quantity controls $\sum_i n_{ij}$ (i.e.~if the same surface enters multiple connected components of the same realization of a realizable strict sign condition then we will over-count, but this is acceptable). The proof is essentially topological.

Let $S\in\sphrs_2$ with $S$ not contained in $W_j$. As in Lemma \ref{sumOfNiLemma}, we can select a large closed ball $B$ so that the number of connected components of $Z_j\backslash W_j$ that $S$ enters is equal to the number of connected components that $S\cap B$ enters. Now, replace $S$ by $S^\prime=\BZ( (f_S+\epsilon)(f_S-\epsilon))$ for $\epsilon>0$ a sufficiently small generic number. Provided $\epsilon$ is sufficiently small, if $S$ meets a connected component $\Delta$ of $Z\backslash W_j$ then $S^\prime$ also meets $\Delta$, since $f_S$ is a continuous function on the (relatively) open set $\Delta,$ so $f_S$ vanishes somewhere on $\Delta$ but does not vanish identically on $\Delta.$ Thus it suffices to count the number of times $S^\prime$ meets a connected component of $Z_j\backslash W_j$.  After replacing $S$ by $S^\prime$ (and recalling that we applied a small generic perturbation to each $Q\in\mathcal Q$), every point in $Z_j\cap W_j\cap S^\prime$ is a point of non-singular intersection.

Now, if $S$ meets a connected component $\Delta$ of $Z_j\backslash W_j$, then one of the following two things must occur:
\begin{enumerate}
 \item\label{boundingCurveEntryFirstItem} $\Delta$ contains (all of) a connected component of $S^\prime\cap Z_j.$
 \item\label{boundingCurveEntrySecondItem} $S^\prime\cap\Delta$ contains a (topological) curve that meets the boundary of $\Delta$ at a point $x\in S^\prime\cap Z_j\cap W_j$. Furthermore, there is at most one other connected component $\Delta^\prime$ for which Item \ref{boundingCurveEntrySecondItem} holds for the same point $x$.
\end{enumerate}

The argument used in Lemma \ref{sumOfNiLemma} shows that Item \ref{boundingCurveEntryFirstItem} can occur at most $O(D_j^2)=O(D_jE_j)$ times. By the real B\'ezout's inequality, $S^\prime\cap Z_j\cap W_j$ contains $O(D_jE_j)$ points of non-singular intersection, and thus Item \ref{boundingCurveEntrySecondItem} can occur at most $O(D_jE_j)$ times. Thus $S^\prime$ can enter at most $O(D_jE_j)$ connected components of $Z_j\backslash W_j$. Since there are at most $n$ surfaces, the result follows.
\end{proof}
\begin{remark}
A similar result to Lemma
\ref{controlOfSphereLevelTwoCellIntersectionsLemma} can be obtained from the recent work of Barone and Basu in \cite{Barone} and Solymosi and Tao in \cite{Solymosi}. 
\end{remark}

Using Lemma \ref{controlOfSphereLevelTwoCellIntersectionsLemma}, we have
\begin{equation}\label{2ndLevelA3InsideCell}
\begin{split}
\incidences(\pts_j \backslash W_j, \sphrs_2) &= \sum_i \incidences(\pts_j\cap \Omega_{ij},\sphrs_2)\\
&\lesssim \sum_i m_{ij}n_{ij}^{1-1/k} + n_{ij}\\
&\leq \Big(\sum_i m_{ij}^k\Big)^{1/k}\Big(\sum_i n_{ij}\Big)^{1-1/k}+n_{ij}\\
&\lesssim \Big(D_jE_j^2\frac{|\pts_j|^k}{(D_jE_j^2)^k}\Big)^{1/k}(nD_j E_j)^{1-1/k}+nD_jE_j\\
&=\frac{|\pts_j|n^{1-1/k}}{E_j^{1-1/k}} + nD_jE_j.
\end{split}
\end{equation}

Our analysis of the second term of \eqref{secondLevelDecomposition} will be the same regardless of whether $j\in\mathcal A_2$ or $\mathcal A_3$. We shall express this bound as a lemma.
\begin{lem}\label{controlOfSurfaceCurveIncidences}
For $j\in\mathcal A_2\cup\mathcal A_3,$ let $Z_j,\ W_j,\ \pts_j,$  and $\sphrs_2$ be as above. Then
\begin{equation}\label{controlOfSurfaceCurveIncidencesEqn}
 \incidences(\pts_j \cap W_j, \sphrs_2)\lesssim n D_jE_j + |\pts_j|.
\end{equation}

\end{lem}
\begin{proof}
We shall write
\begin{equation}
 \incidences(\pts_j\cap W_j, \sphrs_2) = \incidences_1(\pts_j\cap W_j,\sphrs_2)+\incidences_2(\pts_j\cap W_j,\sphrs_2),
\end{equation}
where $\incidences_1$ counts those incidences between points $p\in\pts_j\cap W_j$ and surfaces $S\in\sphrs_2$ such that $p^*$ lies on a 1 (complex) dimensional component of $S^* \cap Z_j^* \cap W_j^*,$ and $\incidences_2$ counts the remaining incidences. To control $\incidences_2$, note that by B\'ezout's inequality (over $\CC$), for each $S\in\sphrs_2,$ $S^*\cap Z_j^*\cap W_j^*$ contains $O(D_jE_j)$ isolated points. Since $|\sphrs_2|\leq n$ we obtain
\begin{equation}\label{controlOfIncidences2}
\incidences_2(\pts_j\cap W_j,\sphrs_2)\lesssim nD_jE_j.
\end{equation}

Thus it remains to control $\incidences_1$. First, we shall replace $\mathcal Q_j$ with a new family of polynomials $\tilde{\mathcal{Q}}_j$ with the following properties:
\begin{enumerate}
 \item\label{WjProperties1} $Z_j\cap W_j\subset Z_j\cap\bigcup_{Q\in\tilde{\mathcal Q}_j}\BZ(Q)$.
 \item\label{WjProperties2} $\sum_{Q\in\tilde{\mathcal Q}_j}\deg Q\leq E_j$.
 \item\label{WjProperties3} Each $Q\in\tilde{\mathcal Q}_j$ is irreducible.
 \item\label{WjProperties4} For each $Q\in\tilde{\mathcal Q}_j$, every irreducible component of $Z_j^*\cap \BZ(Q)^*$ that contains a real point has (complex) dimension 1.
\end{enumerate}
The procedure will be similar to that in the proof of Lemma \ref{allPolysGenerateRealIdeals}: For each $Q\in\mathcal Q_j,$ write $Q=Q_1,\ldots,Q_a$ as a product of  irreducible factors. Discard those factors $Q_b$ with $\BZ(Q_b)\cap Z_j=\emptyset$. Of the remaining factors, place each irreducible factor that generates a real ideal in $\tilde{\mathcal Q}_j$. If $Q_b$ is a factor that does not generate a real ideal then consider $\nabla_v Q_b$ for $v$ a generic vector. By assumption, $Q_b$ does not vanish identically on $Z_j$, but it does vanish at at least one point on $Z_j$. Thus $Q_b$ is not constant on $Z_j$, so $\nabla Q_j$ does not vanish identically on $Z_j$ and hence if $v$ is a generic vector then $\nabla_v Q_b$ does not vanish identically on $Z_j$. Thus we can repeat the above procedure with $\nabla_v Q_b$ in place of $Q$. This process will eventually terminate, and the resulting collection of polynomials $\tilde{\mathcal Q}_j$ has the desired properties; Properties \ref{WjProperties1}--\ref{WjProperties3} are immediate. To obtain Property \ref{WjProperties4}, suppose that $Z_j^*\cap \BZ(Q)^*$ fails to be a complete intersection for some $Q\in\tilde{\mathcal Q}_j$. Then there exists some variety $Y$ that is an irreducible component of both $Z_j^*$ and $\BZ(Q)^*$. by Proposition \ref{factorizationOfComplexVarietyProp} in Appendix \ref{propOfRealVarietiesAppendix}, $\mathcal{R}(Y)$ is an irreducible component of $Z_j$ and $\BZ(Q)$, and thus either $\mathcal{R}(Y)=\emptyset$ or $\mathcal{R}(Y)=Z_j=\BZ(Q)$. The latter is impossible since $Z_j$ and $\BZ(Q)$ have dimension 2, while $Z_j\cap\BZ(Q)$ has dimension at most 1.

Let
\begin{equation*}
\tilde W_j = \bigcup_{Q\in\tilde{\mathcal Q}_j}\BZ(Q).
\end{equation*}
We can write
\begin{equation}\label{decompositionAsIrreducibles}
Z_j^*\cap \tilde W_j^*=\bigcup Y_j
\end{equation}
as a union of irreducible (complex) varieties. By Property \ref{WjProperties4} above, we need only consider those components with (complex) dimension $1$. We shall discard all components that have dimension 2. Let
\begin{align*}
\tilde\pts_j = \{p\in\pts_j\colon\ &\textrm{there exists a (Euclidean) neighborhood $U\subset\CC^3$ of $p^*$ such that}\\
&Z_j^*\cap \tilde W_j^*\cap U\ \textrm{is a (topological) 1--complex-dimensional curve}\}.
\end{align*}

We shall establish several claims.
\begin{enumerate}
\item\label{itm1} $Z_j^*\cap \tilde W_j^*$ is a union of $O(D_jE_j)$ irreducible varieties.
\item\label{itm2} If $p\in\tilde\pts_j$ then $p^*$ lies on at most one of the irreducible component from \eqref{decompositionAsIrreducibles}.
 \item\label{itm3} Let $Y$ be a variety from the above decomposition. If there exist three surfaces $S_1,S_2,S_3\in\sphrs_2$ such that $Y\subset S_i^*,\ i=1,2,3$, then $|\pts_j\cap \mathfrak{R}(Y)|\leq C$.
 \item\label{itm4} If $S\in\sphrs_2$, then there are $O(D_jE_j)$ points $p\notin\tilde\pts_j$ such that $p^*$ is contained in a 1--dimensional component of $S^*\cap Z_j^*\cap W_j^*$.
\end{enumerate}
For Item \ref{itm1} see e.g.~\cite{Fulton}. Item \ref{itm2} follows from the assumption that every variety in the decomposition \eqref{decompositionAsIrreducibles} has dimension $1$. Item \ref{itm3} follows from the requirement that any three surfaces intersect in at most $C$ points. To obtain Item \ref{itm4}, suppose that $D_j\leq E_j$ (if not, we can interchange the roles of $Z_j$ and $W_j$). Note if $p$ satisfies the requirements of Item \ref{itm4}, then $p^*$ is a point of $S^*\cap Z_j^*$ at which $S^*\cap Z_j^*$ fails to be (locally) a 1--dimensional (complex) curve. Thus after a generic rotation of the coordinate axis, the image of $p^*$ under the projection $(x_1,x_2,x_3)\mapsto(x_1,x_2)$ is a singular point of the (complex) plane curve $\BZ(\operatorname{res}_{x_3}(f_S,P_j))^*,$ where $\operatorname{res}_{x_3}$ is the bivariate polynomial obtained by taking the resultant of $f_S$ and $P_j$ in the $x_3$ variable. This curve has degree $O(D_j)$ and thus has $O(D_j^2)=O(D_jE_j)$ singular points.

Now, for each $S\in\sphrs_2$, at most $O(D_jE_j)$ points $p\in\pts_j\backslash\tilde\pts_j$ can contribute to $\incidences_1(\pts_j\cap W_j,\sphrs_2)$, so the total contribution from all surfaces in $\sphrs_2$ is $O(nD_jE_j)$. To control the remaining incidences, use Item \ref{itm3} to write $\{Y_j\}=\{Y_j^\prime\}\sqcup \{Y_j^{\prime\prime}\},$ where the first set consists of varieties that are contained in at most 2 surfaces $S\in\sphrs_2$, and the second consists of varieties that contain at most $C$ points. Each point $p\in\tilde\pts_j$ with $p^*\in\bigcup Y_j^\prime$ can be incident to at most two surfaces, so the total contribution from such points is $O(|\pts_j|)$. On the other hand, by Item \ref{itm1} at most $O(D_jE_j)$ points can be contained in $\mathfrak R(\bigcup Y_j^{\prime\prime})$, so these points can contribute at most $O(nD_jE_j)$ incidences.
\end{proof}
Combining \eqref{2ndLevelA2InsideCell}, \eqref{2ndLevelA3InsideCell}, and  \eqref{controlOfSurfaceCurveIncidencesEqn} and optimizing in $E_j$, we see that our choice of $E_j$ from \eqref{defnOfE} yields the bound

\begin{equation}\label{controlOfLevelTwoCells}
 \incidences(\pts_j , \mathcal S_2)\lesssim |\pts_j|^{\frac{k}{2k-1}}n^{\frac{2k-2}{2k-1}}D_j^{\frac{k-1}{2k-1}}+m_j.
\end{equation}
Summing \eqref{controlOfLevelTwoCells} over all $j\in\mathcal A_2\cup\mathcal A_3$ and noting that $\frac{2k-1}{k}$ and $\frac{2k-1}{k-1}$ are conjugate exponents, we obtain

\begin{equation}\label{controlOfLevelTwoCells2ndEqn}
\begin{split}
\incidences(\pts \cap \bigcup_{\mathcal A_2\cup\mathcal A_3}Z_j, \mathcal S_2) &\lesssim \sum_{\mathcal A_2\cup\mathcal A_3} |\pts_j|^{\frac{k}{2k-1}}n^{\frac{2k-2}{2k-1}}D_j^{\frac{k-1}{2k-1}}+|\pts_j|\\
&\lesssim  n^{\frac{2k-2}{2k-1}}\Big(\sum_j|\pts_j|\Big)^{\frac{k}{2k-1}}\Big(\sum_j D_j\Big)^{\frac{k-1}{2k-1}}+m\\
&\lesssim m^{\frac{k}{2k-1}}n^{\frac{2k-2}{2k-1}}D^{\frac{k-1}{2k-1}}+m.
\end{split}
\end{equation}

Finally, selecting
\begin{equation}\label{valueOfD}
D=m^{\frac{k}{3k-1}}n^{\frac{-1}{3k-1}},
\end{equation}
which by \eqref{mAndMNotTooDifferent} satisfies $D>C$, and combining \eqref{mAndMNotTooDifferent}, \eqref{controlOfLevelOneCells}, \eqref{controlOfLevelOneEmbeddedSpheres}, \eqref{controlOfIncidencesFromA1Summed}, and \eqref{controlOfLevelTwoCells2ndEqn}, we obtain
\begin{equation}\label{sumOfAllTerms}
\begin{split}
\incidences(\pts, \mathcal S)&\lesssim D^2n + m+\frac{mn^{1-1/k}}{D^{1-1/k}}+Dm^{2/3}\\
&\phantom{\lesssim}+nD+m+m^{\frac{k}{2k-1}}n^{\frac{2k-2}{2k-1}}D^{\frac{k-1}{2k-1}}\\
&\lesssim m^{\frac{2k}{3k-1}}n^{\frac{3k-3}{3k-1}}+m^{\frac{2}{3}+\frac{k}{3k-1}}n^{\frac{-1}{3k-1}}+m+n\\
&\lesssim m^{\frac{2k}{3k-1}}n^{\frac{3k-3}{3k-1}}+m+n.
\end{split}
\end{equation}
\end{proof}
\section{Applications}
In \cite{Erdos1,Erdos2}, Erd\H{o}s asked how many unit distances there could be amongst $m$ points in the plane or in $\RR^3$. Theorem \ref{mainThm} yields new bounds for the $\RR^3$ version of this question. Let $\mathcal P$ be a collection of $m$ points in $\RR^3$, and let $\mathcal S$ be a collection of unit spheres centered about the points in $\pts$. We can immediately verify that any three spheres have at most $O(1)$ points in common, so Theorem \ref{mainThm} tells us that there are $O(m^{3/2})$ point-sphere incidences, i.e.
\begin{thm}
The maximum number of unit-distance pairs in a set of $m$ points in $\RR^3$ is $O(m^{3/2}).$
\end{thm}
This is a slight improvement over the previous bound of $
O(m^{3/2}\beta(m))$ from \cite{Clarkson}, where $\beta$ is a very slowly growing function.

As observed in \cite{Clarkson}, theorem \ref{mainThm}, combined with the method outlined in \cite{Chung} can be used to establish bounds on the number of incidences between points and spheres in $\RR^d$. Specifically, we have the following theorem:
\begin{thm}
The maximum number of incidences between $m$ points and $n$ spheres in $\RR^d$ is \begin{equation}
O(m^{d/(d+1)}n^{d/(d+1)}+m+n),
\end{equation}
provided no $d$ of the spheres intersect in a common circle.
\end{thm}
Again, this is a slight improvement (by a $\beta(m,n)$ factor) from the analogous bounds established
in \cite{Clarkson}. See \cite[\S 6.5]{Clarkson} for additional applications of Theorem \ref{mainThm}. In each case, we are able to slightly sharpen the bound from \cite{Clarkson} by removing the $\beta(m)$ factor.

\section{Generalizations to higher dimension}
It is reasonable to ask whether Theorem \ref{mainThm} can be generalized to incidences between points and hypersurfaces in higher dimensions. This task appears to be quite involved, as the necessary algebraic geometry becomes more difficult. In particular, it appears that in order to generalize the proof of Theorem \ref{mainThm} to (say) spheres in $\RR^d$, we need to perform $d-1$ polynomial ham sandwich decompositions, with each successive decomposition performed on the variety defined by the previous decompositions. As $d$ increases, the number of cases to be considered increases dramatically, and certain difficulties such as the failure of the connected components of a complete intersection to themselves be a complete intersection, the failure of an arbitrary complete intersection to be a nonsingular complete intersection, etc.~become increasingly problematic.

One could also consider dimension 2 surfaces in $\RR^d$, $d>3$, and this appears to be more promising. However, the analogues of \eqref{controlOfLevelOneEmbeddedSpheres} and Lemma \ref{controlOfSurfaceCurveIncidences} become more difficult: an algebraic variety of dimension $d-1$ can contain many 2--dimensional surfaces without obvious constraints being imposed on its structure, and in higher dimensions there are more (and more complicated) ways in which varieties can fail to intersect completely. Nevertheless, this is certainly a promising area for future work.
\appendix
\section{Definitions}\label{algebraicGeometryDefinitions}
\begin{defn}Let $\mathcal Q\subset\RR[x_1,\ldots,x_d]$ be a collection of non-zero real polynomials. A \emph{strict sign condition} on $\mathcal Q$ is a map $\sigma\colon Q\to\{\pm 1\}$. If $Q\in\mathcal Q$, we will denote the evaluation of $\sigma$ at $Q$ either by $\sigma_Q$ or $\sigma(Q)$, depending on context. If $\sigma$ is a strict sign condition on $\mathcal Q$ we define its \emph{realization} by
\begin{equation}\label{aSignCondition}
\reali(\sigma,\mathcal Q)=\{x\in\RR^d\colon Q(x)\sigma_Q>0\ \textrm{for all}\ Q\in\mathcal Q\}.
\end{equation}
If $\reali(\sigma,\mathcal Q)\neq\emptyset$ then we say that $\sigma$ is \emph{realizable}. We define
\begin{equation}\label{setOfRealizableSignConditions}
\Sigma_{\mathcal Q}=\{\sigma\colon\reali(\sigma,\mathcal Q)\neq\emptyset\},
\end{equation}
and
\begin{equation}\label{RealizationsOfSignCondition}
\reali(\mathcal Q)=\{\reali(\sigma,\mathcal Q) \colon \sigma\in\Sigma_{\mathcal Q}\}.
\end{equation}
We call $\reali(\mathcal Q)$ the collection of ``realizations of realizable strict sign conditions of $\mathcal Q$.''

If $Z\subset\RR^d$ is a variety, and $\sigma$ is a strict sign condition on $\mathcal Q$, then we can define the \emph{realization of $\sigma$ on Z} by
\begin{equation}\label{aSignConditionOnZ}
\reali(\sigma,\mathcal Q,Z)=\{x\in Z\colon Q(x)\sigma_Q>0\ \textrm{for all}\ Q\in\mathcal Q\},
\end{equation}
and we can define analogous sets
\begin{equation}\label{setOfRealizableSignConditionsOnZ}
\Sigma_{\mathcal Q,Z}=\{\sigma\colon \reali(\sigma,\mathcal Q,Z)\neq\emptyset\},
\end{equation}
and
\begin{equation}\label{RealizationsOfSignConditionOnZ}
\reali(\mathcal Q,Z)=\{\reali(\sigma,\mathcal Q,Z) \colon \sigma\in\Sigma_{\mathcal Q,Z}\}.
\end{equation}
We call $\reali(\mathcal Q,Z)$ the collection of ``realizations of realizable strict sign conditions of $\mathcal Q$ on $Z$.'' Note that if some $Q\in\mathcal Q$ vanishes identically on $Z$ then $\Sigma_{\mathcal Q,Z}=\emptyset$ and thus $\reali(\mathcal Q,Z)=\emptyset$.
\end{defn}
\begin{defn}\label{defnOfRealIdeal}
 An ideal $I\subset\RR[x_1,\ldots,x_d]$ is \emph{real} if for every sequence $a_1,\ldots,a_\ell\in\RR[x_1,\ldots,x_d]$, $a_1^2+\ldots+a_\ell^2\in I$ implies $a_j\in I$ for each $j=1,\ldots,\ell$.
\end{defn}

\section{Properties of real varieties}\label{propOfRealVarietiesAppendix}
The following proposition shows that real principal prime ideals and their corresponding real varieties have some of the ``nice'' properties of ideals and varieties defined over $\CC$.
\begin{prop}[see {\cite[\S 4.5]{Bochnak}} ]\label{propertiesOfPrincipleRealIdealProp}
 Let $(P)\subset\RR[x_1,\ldots,x_d]$ be a principal prime ideal. Then the following are equivalent:
 \begin{enumerate}
 \item $(P)$ is real.
 \item\label{nullz} $(P) = \mathbf I (\BZ(P))$.
 \item\label{smallDim} $\dim (\BZ (P))=d-1$.
 \item $\nabla P$ does not vanish identically on $\BZ(P)$.
 \item\label{noSignChange} The sign of $P$ changes somewhere on $\RR^d$.
 \end{enumerate}
 \end{prop}

\begin{defn}
If $P\subset\RR[x_1,\ldots,x_d]$ is a polynomial and $P=P_1,\ldots, P_\ell$ is its factorization, we define $\hat P$ to be the polynomial obtained by removing those irreducible components that generate ideals that aren't real. If every irreducible component of $P$ generates an ideal that is not real, then we define $\hat P = 1$.
\end{defn}
\begin{example}
Let $P=(x_1^2+x_2^2+x_3^2-1)(x_1^2+x_2^2).$ Then $\hat P=x_1^2+x_2^2+x_3^2-1$. Geometrically, if $\hat P\neq0$, then $\BZ(P)$ is a $(d-1)$--dimensional (real) variety, but some of the components of $\BZ(P)$ may have dimension less than $d-1$. $\hat P$ keeps those factors that generate components that have dimension $d-1$.
\end{example}

The existence of polynomials that do not generate real ideals complicates our analysis, but since the zero sets of such polynomials have codimension at least 2, we can ignore them when we are computing the number of times a surface meets the realization of a realizable strict sign condition of a family of polynomials. The following theorem helps make this statement precise.
\begin{thm}\label{ignoreNonRealComponents}
Let $\mathcal Q\subset\RR[x_1,\ldots,x_d],\ d\geq 3$ be a collection of real polynomials and let $\hat{\mathcal Q}=\{\hat Q\colon Q\in\mathcal Q\}\backslash\{0\}$. Then there exists a bijection
\begin{equation*}
\tau\colon \reali(\mathcal Q)\to \reali(\hat{\mathcal Q})
\end{equation*}
such that
\begin{equation}\label{ignoreNonRealComponentsEqn1}
X \subset \tau(X)\ \textrm{for every}\ X\in \reali(\mathcal Q).
\end{equation}

Similarly, if $Z=\BZ(P)$ where $P\in\RR[x_1,\ldots,x_d]$ generates a real ideal and no polynomial $Q\in\mathcal Q$ vanishes identically on $Z$, then there exists a bijection
\begin{equation*}
\tau\colon \reali(\mathcal Q,Z)\to \reali(\hat{\mathcal Q},Z)
\end{equation*}
such that \eqref{ignoreNonRealComponentsEqn1} holds with $\reali(\mathcal Q,Z)$ in place of $\reali(\mathcal Q).$
\end{thm}

\begin{proof}
First, by Item \ref{noSignChange} of Proposition \ref{propertiesOfPrincipleRealIdealProp}, for each $Q\in\mathcal Q$ we have that $Q/\hat Q\geq 0$ or $Q/\hat Q\leq 0$ on all of $\RR^d$. Choose $\varepsilon_Q\in\{\pm1\}$ so that $\varepsilon_Q Q/\hat Q\geq 0$. Now, note that if there exist $Q_1,Q_2\in\mathcal Q$ with $\hat Q_1=\hat Q_2$ and if $\sigma$ is a strict sign condition on $\mathcal Q$, then either $\varepsilon_{Q_1} \sigma(Q_1)=\varepsilon_{Q_2}\sigma(Q_2)$ or $\reali(\sigma,\mathcal Q)=\emptyset$. Thus if $\sigma$ is a realizable strict sign condition on $\mathcal Q$, then we can define $\hat\sigma\colon\hat{\mathcal Q}\to\{\pm1\}$ by $\hat\sigma(T)=\varepsilon_Q\sigma(Q)$, where $Q\in \mathcal Q$ satisfies $T=\hat Q$, and $\hat\sigma$ is well-defined.

We shall show that the map $\Sigma_{\mathcal Q}\to\Sigma_{\hat{\mathcal Q}},\ \sigma\mapsto\hat\sigma$ is a bijection. To prove injectivity, note that if distinct $\sigma_1,\sigma_2$ both map to the same element $\hat\sigma$, then $\varepsilon_Q\sigma_1(Q)=\varepsilon_Q\sigma_2(Q)$ for all $Q\in\mathcal Q$, so clearly $\sigma_1=\sigma_2$. To establish surjectivity, note that each $\sigma_1\in\Sigma_{\tilde Q}$ has a pre-image under the map $\sigma\mapsto\hat\sigma$. Thus every element of $\Sigma_{\tilde Q}$ may be written as $\hat\sigma$ for some strict sign condition $\sigma$ on $\mathcal Q$. All that we must establish is that $\sigma$ is realizable. For each $Q\in\mathcal Q,$ we have
\begin{equation}\label{signConditionMinusTildeCodim2}
\dim\big(\{\hat Q>0\}\backslash \{\varepsilon_Q Q >0\}\big)\leq d-2,
\end{equation}
(see i.e.~\cite{Bochnak} for the dimension of a semi-algebraic set). On the other hand, the realization of each realizable strict sign condition of $\hat{\mathcal Q}$ has dimension $d$. Thus if $\reali(\hat\sigma,\hat{\mathcal Q})\neq\emptyset$ then $\reali(\sigma,\mathcal Q)$ can be written as a (non-empty) dimension $d$ semi-algebraic set minus a dimension $d-2$ semi-algebraic set, and in particular, $\reali(\sigma,\mathcal Q)\neq\emptyset.$

Thus the map $\reali(\mathcal Q)\mapsto\reali(\hat{\mathcal Q})$ which takes $X=\reali(\sigma,\mathcal Q)\mapsto\reali(\hat\sigma,\hat{\mathcal Q})$ is well-defined and is a bijection. Now, note that by Items \ref{smallDim} and \ref{noSignChange} of Proposition \ref{propertiesOfPrincipleRealIdealProp}, $\{\varepsilon_QQ>0\}\subset\{\hat Q>0\}$,
and similarly with ``$>$'' replaced by ``$<$''). Thus
\begin{equation}\label{XQsigmaSubset}
\reali(\sigma,\mathcal Q)\subset\reali(\hat\sigma,\hat{\mathcal Q}),
\end{equation}
so \eqref{ignoreNonRealComponentsEqn1} holds.

The same arguments establish the second part of the theorem. The only new thing that must be verified is that the map $\Sigma_{\mathcal Q,Z}\to\Sigma_{\hat{\mathcal Q},Z},\ \sigma\mapsto\hat\sigma$ is onto. However, this is established by \eqref{signConditionMinusTildeCodim2} plus the fact that the realization of each realizable strict sign condition of $\mathcal Q$ on $Z$ has dimension $d-1$.
\end{proof}
\begin{cor}\label{surfaceEnteringStrictSignConditions}
Let $S\subset\RR^3$ be a smooth surface, let $\mathcal Q$ be a collection of polynomials, and let $\hat{\mathcal Q}$ be as in Theorem \ref{ignoreNonRealComponents}. Then
\begin{equation}
|\{X\in\reali(\mathcal Q)\colon X\cap S\neq\emptyset\}|\leq |\{X\in\reali(\hat{\mathcal Q})\colon X\cap S\neq\emptyset\}|.
\end{equation}
Similarly, let $S\subset\RR^3$ be a smooth surface, let $Z=\BZ(P)$ where $P\in\RR[x_1,x_2,x_3]$ generates a real ideal, let $\mathcal Q$ be a collection of polynomials, none of which vanish identically on $Z$, and let $\hat{\mathcal Q}$ be as in Theorem \ref{ignoreNonRealComponents}. Then
\begin{equation}
|\{X\in\reali(\mathcal Q,Z)\colon X\cap S\neq\emptyset\}|\leq |\{X\in\reali(\hat{\mathcal Q},Z)\colon X\cap S\neq\emptyset\}|.
\end{equation}
\end{cor}

As noted in Section \ref{obstructions}, the number of intersection points of a collection of real polynomials may exceed the product of their degrees, even if those polynomials intersect completely. Over $\CC$ things are much better behaved, so there will be times when we will wish to embed everything into $\CC$. The following proposition relates the properties of a variety defined over $\RR$ and the corresponding variety defined over $\CC$:

\begin{prop}[see {\cite[\S 10]{Whitney}}]\label{factorizationOfComplexVarietyProp}
 Let $Z\subset\RR^d$ be a real variety and let $Z_1^*,\ldots,Z^*_\ell$ be the irreducible components of $Z^*$. Then $\mathfrak{R}(Z_1^*),\ldots,\mathfrak{R}(Z_\ell^*)$ are the irreducible components of $Z$.
\end{prop}

\section{Proof of Theorem \ref{variantHamSandwichThm}}\label{proofOfPolyHamAppendix}

 \label{variantHamSandwichThmAppendix} Our proof of Theorem \ref{variantHamSandwichThm} will be similar to the original proof of the discrete polynomial ham sandwich theorem in \cite[\S4]{Guth}, which can be stated as follows:
\begin{prop}[Discrete polynomial ham sandwich theorem]\label{discreteHamSandProp}
Let $V\subset\RR[x_1,\ldots,x_d]$ be a vector space of dimension $\ell$, and let $F_1,\ldots,F_\ell\subset\RR^d$ be finite families of points. Then there exists a polynomial $P\in V$ such that
\begin{equation*}
\begin{split}
&|F_j \cap \{x\in\RR^d\colon P(x)>0\}|\leq |F_j|/2,\ \textrm{and}\\
&|F_j\cap \{x\in\RR^d\colon P(x)<0\}|\leq|F_j|/2,\ j=1,\ldots,\ell.
\end{split}
\end{equation*}
\end{prop}
Proposition \ref{discreteHamSandProp} is proved in \cite{Guth} only in the special case where $V$ is the vector space of all polynomials of degree at most $e$ (where $e$ is chosen large enough to ensure that $V$ has the required dimension). However, the proof carries over verbatim to the general case where $V$ is arbitrary. To prove Theorem \ref{variantHamSandwichThm}, we will iterate the following lemma:

\begin{lem}\label{polyHamLemma}
Let $Z=\BZ(P)\subset\RR^d$ for $P$ an irreducible polynomial of degree $D$ such that $(P)$ is a real ideal. Let $E>0$, and let $F_1,\ldots, F_\ell,\ \ell= c \min(E^d, DE^{d-1})$ be finite families of points in $\RR^d$, with $F_j\subset Z$ for each $j$. Then provided $c$ is sufficiently small (depending only on $d$), there exists a polynomial $Q$ of degree at most $E$ that does not vanish identically on $\BZ(P)$  such that
\begin{equation}\label{propertiesOfQ}
\begin{split}
&|F_j \cap \{x\in\RR^d\colon Q(x)>0\}|\leq |F_j|/2,\ \textrm{and}\\
&|F_j\cap \{x\in\RR^d\colon Q(x)<0\}|\leq|F_j|/2,\ j=1,\ldots,\ell.
\end{split}
\end{equation}
\end{lem}
\begin{proof}
Let $\RR[x_1,\ldots,x_d]_{\leq E}$ be the vector space of all polynomials in $d$ variables of degree at most $E$, and let $(P)_{\leq E}$ be the vector space of all polynomials in the ideal $(P)$ that have degree at most $E$ (of course, if $E < \deg P$ then $(P)_{\leq E}=0$). We have
\begin{equation*}
\dim(\RR[x_1,\ldots,x_d]_{\leq E}) -\dim ((P)_{\leq E}) >  c\min(E^d, DE^{d-1})
\end{equation*}
for some (explicit) constant $c$ depending only on the dimension $d$. Thus, we can find a vector space $V\subset \RR[x_1,\ldots,x_d]_{\leq E}$ with $\dim(V)> c\min(E^d, DE^{d-1})$ such that $V\cap (P)_{\leq E}=0$. By Proposition \ref{discreteHamSandProp}, we can find a polynomial $Q\in V$ satisfying \eqref{propertiesOfQ}. Since $Q\in\RR[x_1,\ldots,x_d]_{\leq E}$ but $Q\notin (P)_{\leq E}$, we have $Q\notin (P)$. Since $P$ is irreducible and generates a real ideal, by Item \ref{nullz} of Proposition \ref{propertiesOfPrincipleRealIdealProp}, $Q$ does not vanish identically on $\BZ(P)$.
\end{proof}
\begin{proof}[Proof of Theorem \ref{variantHamSandwichThm}]
Use Lemma \ref{polyHamLemma} to find a polynomial $Q_1$ of degree $O(1)$ such that
\begin{equation*}
\begin{split}
&|\{x\in\RR^d\colon Q(x)>0\}\cap\pts |\leq |\pts|/2,\\
&|\{x\in\RR^d\colon Q(x)<0\}\cap\pts|\leq |\pts|/2.
\end{split}
\end{equation*}
Let $\mathcal Q_1=\{Q_1\}$. For each $i=2,\ldots, t,$ with
\begin{equation}\label{valueOfT}
t = \lceil \log_2(DE^{d-1})\rceil,
\end{equation}
use Lemma \ref{polyHamLemma} to find a polynomial $Q_{i}$ with
\begin{equation*}
\deg(Q_i)\lesssim \max \big( (2^i/D)^{1/(d-1)}, 2^{i/d}\big)
\end{equation*}
such that for each $\sigma\in\Sigma_{\mathcal Q_{i-1}}$ we have
\begin{equation}\label{cutSignConditionsInHalf}
\begin{split}
\Big|\{x\in\RR^d\colon Q_i(x)>0\}\cap \big(\pts\cap\reali(\sigma,\mathcal Q_{i-1})\big)\Big|&\leq\frac{1}{2}|\pts\cap\reali(\sigma,\mathcal Q_{i-1})|,\\
\Big|\{x\in\RR^d\colon Q_i(x)<0\}\cap \big(\pts\cap\reali(\sigma,\mathcal Q_{i-1})\big)\Big|&\leq\frac{1}{2}|\pts\cap\reali(\sigma,\mathcal Q_{i-1})|.
\end{split}
\end{equation}
Some of the above sets may be empty, but this does not pose a problem. Let $\mathcal Q_i = \mathcal Q_{i-1}\cup\{Q_i\}$.

None of the polynomials in $\mathcal Q =\mathcal Q_t$ vanish on $P$, so Item \ref{variantHamSandwichThmItm3} of the theorem is satisfied. Since $E>\rho D$ we have
\begin{align*}
\sum_{\mathcal Q}\deg Q &\lesssim \sum_{i=1}^t (2^i/D)^{1/(d-1)} + \sum_i 2^{i/d}\\
& \lesssim (DE^{d-1}/D)^{1/(d-1)} + (DE^{d-1})^{1/d}\\
& \lesssim E,
\end{align*}
which satisfies Item \ref{variantHamSandwichThmItm2}.
By \eqref{cutSignConditionsInHalf}, for each $\sigma\in\Sigma_{\mathcal Q}$,
\begin{equation}
\begin{split}
|\pts\cap\reali(\sigma,\mathcal Q)|&\lesssim 2^{-t}|\pts|\\
&\lesssim\frac{|\pts|}{DE^{d-1}},
\end{split}
\end{equation}
which satisfies Item \ref{variantHamSandwichThmItm4}. Finally, Item \ref{variantHamSandwichThmItm1} follows from \eqref{valueOfT}.
\end{proof}
\bibliographystyle{amsplain}

\end{document}